\documentclass[a4paper,11pt]{article}

\usepackage[latin1]{inputenc}
\usepackage[T1]{fontenc}
\usepackage{amssymb,amsmath,amsthm,amscd,mathrsfs}
\usepackage[all]{xy}

\usepackage{color}

\def\Q{\mathbb{Q}}
\def\G{\mathbb{G}}
\def\P{\mathbb{P}}
\def\Z{\mathbb{Z}}
\def\C{\mathbb{C}}

\def\kk{K3 surface }
\def\ks{K3 surfaces }

\def\ISM{irreducible symplectic manifold }
\def\ISMs{irreducible symplectic manifolds }

\def\LF{Lagrangian fibration }
\def\LFs{Lagrangian fibrations }

\def\EE{\mathcal{E}}
\def\FF{\mathcal{F}}
\def\GG{\mathcal{G}}
\def\HH{\mathcal{H}}
\def\JJ{\mathcal{J}}
\def\LL{\mathcal{L}}
\def\NN{\mathcal{N}}
\def\OO{\mathcal{O}}
\def\PP{\mathcal{P}}
\def\SS{\mathcal{S}}

\newcommand{\Aut}{\operatorname{Aut}\nolimits}
\newcommand{\Bl}{\operatorname{Bl}\nolimits}
\newcommand{\Ext}{\operatorname{Ext}\nolimits}
\newcommand{\Imm}{\operatorname{Im}\nolimits}

\newcommand{\Indet}{\operatorname{Indet}\nolimits}
\newcommand{\Fix}{\operatorname{Fix}\nolimits}
\newcommand{\PAut}{\operatorname{PAut}\nolimits}
\newcommand{\Pic}{\operatorname{Pic}\nolimits}
\newcommand{\supp}{\operatorname{supp}\nolimits}
\newcommand{\tr}{\operatorname{tr}\nolimits}

\newtheorem*{thmm}{Theorem}
\newtheorem{thm}[subsection]{Theorem}
\newtheorem{df}[subsection]{Definition}
\newtheorem{lm}[subsection]{Lemma}
\newtheorem{cor}[subsection]{Corollary}

\theoremstyle{remark}
\newtheorem{remark}[subsection]{Remark}

\begin{document}

\title{\bf A singular symplectic variety of dimension 6 with a Lagrangian Prym fibration}

\author{Tommaso \textsc{Matteini}}

\date{}

\maketitle

\begin{center}
IRMA - Universit\'e de Strasbourg \\
7 rue Ren\'e Descartes, 67000 Strasbourg - France\\
Telephone: +33 (0)3 68 85 01 53 \\
E-mail address:  \texttt{matteini@unistra.fr}
\end{center}

\begin{abstract}
A projective symplectic variety $\PP$ of dimension 6, with only finite quotient singularities, $\pi(\PP)=0$ and $h^{(2,0)}(\PP_{smooth})=1$, is described as a relative compactified Prym variety of a family of genus 4 curves with involution. 
It is a \LF associated to a \kk double cover of a generic cubic surface. 
It has no symplectic desingularization. 
\end{abstract}

\begin{center}
\end{center}

\newpage

\section*{Introduction}

In this paper, we work in the setting of complex projective varieties. 

An \ISM is a simply connected projective manifold which has a unique holomorphic symplectic structure. 
A \LF is a surjective proper morphism with connected fibers from a projective symplectic manifold onto a  projective manifold such that the generic fiber is a connected Lagrangian submanifold, which is an abelian variety by Arnold-Liouville theorem. 
\LFs play a central role in holomorphic symplectic geometry, as they appear in all the (few) known deformation classes of irreducible symplectic manifolds. 
Moreover a surjective proper morphism with connected fibers from an \ISM onto a projective manifold of smaller positive dimension is a \LF by a theorem of Matsushita \cite{Ma}, and the base of the fibration is  a projective space by a result of Hwang \cite{Hw}. 

The problem of determining \ISMs as (compactified) families of abelian varieties over projective spaces naturally arises. 
A famous example, described by Beauville in \cite{B1}, is the relative compactified Jacobian $\JJ^H_C$ of a linear system $|C|$ on a \kk $S$ with respect to a polarization $H$ of $S$. 
It is the moduli space of $H$-semistable sheaves on $S$ of rank zero with $c_1=[C]$ and $\chi=1-g$. 
Its smooth locus inherits a symplectic structure from $S$ described by Mukai  \cite{Mu}. 
The fibration is given by the (Fitting) support map $\supp: \JJ^H_C \to |C|$. 
When the linear system contains only integral curves, then $\JJ^H_C$ is smooth and is deformation equivalent to a Hilbert scheme of points on a K3 surface.  
This gives the so called Beauville-Mukai integrable system. 
Basing on Sawon's work \cite{S2}, it seems plausible that all the \LFs which are relative compactified Jacobians of families of curves, are Beauville-Mukai integrable systems. 

In order to look for \LFs with more general abelian varieties as fibers, the next natural candidates are families of Prym varieties. 
Markushevich and Tikhomirov deal with this problem in an example \cite{MT}, suggesting the following general construction. 
Consider a \kk $S$ with an involution $\tau$ acting non-trivially on the symplectic form (a so called antisymplectic involution) and take a linear system of curves $|C|^{\tau}$ on $S$ invariant with respect to $\tau$. 
There is a natural fibration over the locus of smooth curves of $|C|^{\tau}$ given by the relative Prym variety. 
It can be interpreted as an open subset of a connected component of the fixed locus of a rational involution $\eta$ preserving the symplectic form on $\JJ^H_C$. 
Hence it inherits the symplectic structure of $\JJ^H_C$. 
When $\eta$ extends to a regular involution on $\JJ^H_C$, the connected component $\PP^H_C$ of the fixed locus is a relative compactified Prym variety. 
This is the case when $H=C$. 
Unfortunately, due to the non-generic choice of the polarization, $\JJ^C_C$ has singular points, corresponding to strictly semistable sheaves, necessarly supported on non-integral curves of $|C|$. 
Moreover, also $\PP^C_C$ is singular along $\eta$-invariant strictly semistable sheaves, necessarly supported on non-integral curves of $|C|^{\tau}$. 
It is natural to ask if there exists a desingularization of $\PP^C_C$ which is symplectic. 

In \cite{MT}, Markushevich and Tikhomirov describe a relative compactified Prym variety of dimension 4 (in their notation $\mathbf{\PP}^0$, Definition 3.3 and paragraph below \cite{MT}).  They consider as $S$ a \kk  double cover of a generic Del Pezzo surface of degree 2, as $\tau$ the corresponding involution, and as $C$ a generic $\tau$-invariant curve in the linear system of the pullback of the anticanonical divisor. 
They characterize the non-integral curves of $|C|^{\tau}$: 28 reduced reducible curves of type $C_1 \cup C_2$, where $C_1, C_2$ are smooth rational curves meeting transversely in 4 points (Lemma 1.1 \cite{MT}). 
There are 28 $\eta$-invariant strictly semistable sheaves, supported on these 28 curves. 
They describe these 28 isolated singularities of $\PP^C_C$ using the Kuranishi model of the relative Jacobian: their analytic type is $(\C^4/\pm 1,0)$ (Theorem 3.4 (i) \cite{MT}). 
As these are $\Q$-factorial and terminal singularities, $\PP^C_C$ has no symplectic desingularization. 

Nevertheless, their example is still interesting for several reasons. 
Firstly, the fibration in Prym surfaces is Lagrangian and its generic fiber has a polarization type $(1,2)$ (Theorem 3.4 (ii) \cite{MT}): no other example of this type is known. 
Secondly, the Euler characteristic of $\PP_C^C$ is calculated (see Remark  \ref{error}), which is an important topological invariant. 
Thirdly, $\PP^C_C$  is simply connected (Proposition 5.4 \cite{MT}) and its smooth locus has $h^{(2,0)}=1$ (it follows from Lemma 5.2 \cite{MT}). 
Namikawa  has proven some results in \cite{Na} which suggest to extend the theory of \ISMs to a larger class of singular symplectic varieties, in particular to $\Q$-factorial projective varieties $X$ with terminal singularities, $\pi(X)=0$ and $h^{(2,0)}(X_{smooth})=1$. 
The Beauville-Bogomolov form can still be defined (Theorem 8 (2) \cite{Na}), and the Local Torelli Theorem still holds (Theorem 8 (3) \cite{Na}). 
Following this research direction, Menet  has determined in \cite{Me} the Beauville-Bogomolov form of the example \cite{MT}. 

Due to these considerations, it is interesting to explore the relative Prym construction in other explicit cases.  
In this paper, we take inspiration from \cite{MT} to focus on a relative compactified Prym variety of dimension 6, considering as $S$ a \kk double cover of a generic Del Pezzo of degree 3 (i.e. a cubic surface), as $\tau$ the corresponding involution, and as $C$ a generic $\tau$-invariant curve in the linear system of the pullback of the anticanonical divisor. 
We denote simply by $\PP$ this relative compactified Prym variety of dimension 6 and by $\JJ$ the corresponding relative Jacobian of dimension 8. 

The main result of this paper is the following:

\begin{thmm}
Let $S$ be a \kk double cover of a generic cubic surface $Y$, $\tau$ the corresponding involution and $\pi$ the corresponding morphism. 
Let $C$ be a generic curve in $\pi^*|-K_{Y}|$. 
Then 
\begin{itemize}
\item[$i)$] $\PP$ is a singular symplectic variety of dimension 6,  which does not admit any symplectic desingularization. 
Its singular locus $Sing(\PP)$ coincides with the locus of $\eta$-invariant strictly semistable sheaves of $\JJ$, and it is the union of 27 singular \ks associated to the 27 lines on $Y$. 
Each \kk has 5 $A_1$-singularities and each singular point is in the intersection of 3 K3 surfaces. 
A smooth point of $Sing(\PP)$ is a singularity of $\PP$ of analytic type $\C^2 \times (\C^4/ \pm 1)$. 
A singular point of $Sing(\PP)$ is a singularity of $\PP$ of analytic type $\C^6 / \Z_2 \times \Z_2$, where the action of $\Z_2 \times \Z_2$ is given by $\langle (1,1,-1,-1,-1,-1),(-1,-1,1,1,-1,-1) \rangle$. 
\item[$ii)$] $\supp: \PP \to |C|^{\tau}$ is a \LF with as the generic fiber an abelian 3-fold with polarization type $(1,1,2)$.  
\item[$iii)$] $\PP$ is simply connected and $h^{(2,0)}(\PP_{smooth})=1$.
\item[$iv)$] $\chi(\PP)=2283$. 

\end{itemize}
\end{thmm}

\begin{proof}
\begin{itemize}
\item[$i)$] Corollary \ref{nosympl}, Theorem \ref{singgen}, Theorem \ref{singpar}, Corollary \ref{singloc}. 
\item[$ii)$] Theorem \ref{Psympl}. 
\item[$iii)$] Corollary \ref{h20} and Theorem \ref{pi1}.  
\item[$iv)$] Theorem \ref{ecchila}. 
\end{itemize}
\end{proof}

In Section 1 we describe the \kk double cover of the generic cubic surface, and we determine the non-integral curves of the pullback of the anticanonical linear system. 
In Section 2 we define the relative compactified Prym variety $\PP$ associated to the anticanonical linear system of the generic cubic surface. 
In Section 3 we characterise the singular locus and the singularities of $\PP$, using the Kuranishi map. 
We show that $\PP$ has terminal $\Q$-factorial singularities, which implies that it does not admit any symplectic resolution. 
In Section 4 we prove that $\PP$ is simply connected and the $H^{(2,0)}$ of its smooth locus is generated by the symplectic form, essentially constructing a rational double cover map from $S^{[3]}$ to $\PP$. 
In Section 5 we compute the Euler characteristic of $\PP$, using the fibration structure.  

\section*{Acknowledgements} 
I thank Dimitri Markushevich, my PhD supervisor, for all his help. 
I worked on this paper at the International School of Advanced Studies of Trieste, at the Paul Painlev\'e Laboratory of the University of Lille1, at the Max Planck Institute for Mathematics of Bonn and at the Institute of Advanced Mathematical Research of the University of Strasbourg: I thank these institutions for their hospitality. 
I also acknowledge a partial support by the Research Network Program GDRE-GRIFGA and by the Labex CEMPI (ANR-11-LABX-0007-01).

\section{\kk double cover of a generic cubic surface}

In this section we describe the \kk $S$ double cover of the generic cubic surface $Y$, and we determine the non-integral curves of $\pi^*|-K_{Y}|$, the pullback of the anticanonical linear system of $Y$.

$\\ $

\begin{lm}\label{sy3}
Let $S$ be a \kk double cover of a generic cubic surface $Y$, $\tau$ the corresponding involution and $\pi$ the corresponding morphism. 

Then $\pi^*(-K_{Y})$ embeds $S$ in $\P^4$ as the intersection of a quadric 3-fold 
\begin{align}\label{y3qu}
Z_2: F_2=x_4^2+f_2=0 \mbox{ with } f_2 \in \C[x_0,x_1,x_2,x_3]_2
\end{align}
and a cubic cone with vertex $p_0=(0,0,0,0,1)$ 
\begin{align}
\label{y3cu}
Z_3: F_3=0 \mbox{ with } F_3 \in \C[x_0,x_1,x_2,x_3]_3.
\end{align}
Moreover $\tau$ is given by $x_4 \mapsto -x_4$, $\pi$ is the restriction of the projection from the point $p_0$ onto the hyperplane $H_4: x_4=0$ and $Y=Z_3 \cap H_4$. 
\end{lm}

\begin{proof}
By the general theory of \ks with antisymplectic involutions, the moduli space of \ks double cover of Del Pezzo surfaces of degree 3 has dimension $13$ (see Section 4 or Subsection 5.6 \cite{M}). 
Hence a \kk intersection of (\ref{y3qu}) and (\ref{y3cu}) represents a generic point of the moduli space, because
$$\dim |\OO_{\P^3}(2)|+\dim |\OO_{\P^3}(3)|-\dim PGL(4)=13.$$
\end{proof}

\begin{lm}\label{blabla}
Let $S$ be a \kk double cover of a generic cubic surface $Y$. 
Then the non-integral curves of $\pi^*|-K_{Y}|$ are parametrized by the 27 lines dual to the 27 lines on $Y$. 
The points lying on only one dual line represent reduced curves of type $C_1 \cup C_2$, where $C_i$ are smooth curves of genus respectively 0 and 1 intersecting transversely in 4 points. 
The remaining 45 points represent reduced curves of type $C^1 \cup C^2 \cup C^3$, where $C^i$ smooth rational curves, intersecting transversely in pairs in 2 points.  
\end{lm}

\begin{proof} 
The non-integral curves of $|-K_Y|$ are parametrized by the 27 lines dual to the 27 lines on $Y$. 
The points lying on only one dual line represent curves which are unions of a conic and a line. 
The remaining points, which are 45 by the configuration of the 27 lines (see V.4 \cite{Ha}), are unions of 3 lines. 
The description of the corresponding curves of $\pi^*|-K_{Y}|$ follows immediately. 

It remains to exclude the existence of other non-integral curves in $\pi^*|-K_{Y}|$. 
If $C'$ is an irreducible curve of $|-K_Y|$ such that the corresponding curve $C$ in $\pi^*|-K_{Y}|$ is reducible, then $C'$ is cut by a hyperplane $H$ totally tangent to the branch locus $B$ of $\pi$. 
Since the arithmetic genus of $C$ is 4, $C$ is the union of two smooth curves of genus 1 meeting in 3 points. 
But $C= S \cap \langle H,p_0\rangle \subset Z_2 \cap \langle H,p_0 \rangle$, the latter intersection being a quadric in $\P^3$. 
A smooth genus 1 curve is in the linear system $|\OO(2)|$ on a quadric surface, while $C$ belongs to $|\OO(3)|$, absurd.

\end{proof}

\section{Construction of $\PP$}

In this section we define the main subject of the paper: the relative compactified Prym variety $\PP$ associated to the anticanonical linear system of the generic cubic surface.

$\\ $

We start by recalling the notion of Prym variety. 

\begin{df}
Let $\xymatrix{ C \ar@(dl,ul)[]^{\tau} \ar[r] & C'}$ be a double cover of smooth curves $C, C'$ of genus $g, g'$ respectively. 
The Prym variety $P(C,\tau)$ of the double cover is the connected component $\Fix^0(-\tau^*)
 \subset J(C)$ of the fixed locus of $-\tau^*$ containing zero. 
It inherits a natural polarization from $J(C)$, the restriction of $\Theta_C$.
\end{df}

From now on we fix a  \kk $S$ double cover of a generic cubic surface $Y$, with corresponding involution $\tau$, and a generic curve $C$ in $\pi^*|-K_{Y}|$.  
Let $\JJ:=\JJ_C^C$ be the relative compactified Jacobian of $|C|$, which is the moduli space of $C$-semistable sheaves of rank $0$, first Chern class $[C] \in H^2(S,\Z)$ and Euler characteristic $-3$. 

In order to define $\PP$, we introduce the relative version of the involutions $-1$ and $\tau^*$. 

\begin{lm}\label{taustar}
$\tau$ induces a regular involution $\tau^*$ on $\JJ$. 
\end{lm}

\begin{proof}
This follows directly from the fact that the polarization $C$ is $\tau$-invariant.
\end{proof}

\begin{lm}\label{defj}
Let $j$ be the involution of $\JJ$ defined by 
$$j(\FF):=\EE xt^1_S(\FF,\OO _S(-C)).$$ 
Then
\begin{itemize}
\item[$i)$] $j$ coincides with $-1$ fiberwise over the locus of smooth curves of ${\pi^*|-K_{Y}|}$;
\item[$ii)$] $j$ is a regular involution.
\end{itemize} 
\end{lm}

\begin{proof}
\begin{itemize}
\item[$i)$] 
Since $C$ is generic, $-1$ is defined as $\HH om_S(\_,\OO _C)$ on $J(C)$, which is the fiber of $\supp$ over the point $C$. 

We prove that $\EE xt^1_S(\_,\OO _S(-C)) \cong \HH om_S(\_,\OO _C)$.

Let $\{V_i\}$ be an open covering of $S$ such that local isomorphisms $\FF|_{V_i} \cong \OO_C|_{V_i}$ hold. 
Applying the functor $\HH om_S(\_,\OO_S(-C))$ to the short exact sequence 
$$0 \to \OO_S(-C) \to \OO_S \to \OO_C \to 0,$$
we get, from the definition of $\EE xt^1$, the canonical isomorphisms
$$\HH om_{S}(\OO_C|_{V_i},\OO_C|_{V_i})=\EE xt^1_{S}(\OO _C|_{V_i},\OO _S(-C)|_{V_i}).$$
We conclude by gluing together these isomorphisms.

\item[$ii)$] First we observe that $c_1(j(\FF))=[C]$ and $\chi(j(\FF))=-3$ for any $\FF \in \JJ$. 

Indeed, $c_1(\EE xt^1_S(\FF,\OO_S))=c_1(j(\FF))$, because tensoring by a line bundle does not change the first Chern class of a $1$-dimensional sheaf, and  $c_1(\EE xt^1_S(\FF,\OO_S))=c_1(\FF). \\ $ 
Moreover 
\begin{equation}\label{jff}
j(\FF)=\FF^* \otimes \OO_S(-C) \otimes \NN_{C/S}, \quad
\NN_{C/S}=\OO_S(C)|_C=\omega_C
\end{equation}
thus by the Hirzebruch-Riemann-Roch theorem,
\begin{equation}\label{chi}
\chi(j(\FF))=
-\chi(\FF)-C \cdot c_1(\FF)
\end{equation}
and we conclude.

Then we show that $j$ preserves the $C$-semistability.

Applying $j$ to the exact sequence $0 \to \GG \to \FF$, we obtain $j(\FF) \to j(\GG) \to 0$. 
Hence there is a 1-1 correspondence between subsheaves of $\FF$ and quotient sheaves of $j(\FF)$. 
Moreover, by (\ref{jff}) and (\ref{chi}) applied to $\GG$ instead of $\FF$, we get
\begin{equation}\label{chiave}
\mu_C(j(\GG)):=\frac{\chi(j(\GG))}{c_1(j(\GG)) \cdot C}= 
 - \mu_C(\GG)- 1. 
\end{equation}
Thus $\mu_C(j(\FF)) \leq \mu_C(j(\GG))$ is equivalent to $\mu_C(\GG) \leq \mu_C(\FF)$.
\end{itemize} 
\end{proof}

\begin{df}\label{defprym}
The relative Prym variety associated to $C$ is 
\begin{equation*}
\PP:=\Fix^0(\eta) \subset \JJ,
\end{equation*} 
the connected component, containing the zero section, of the fixed locus of 
\begin{equation*}
\eta:=j \circ \tau^* \quad \xymatrix{ \JJ \ar@(dl,ul)}
\end{equation*}
where $j$ is defined in Lemma \ref{defj} and $\tau^*$ in Lemma \ref{taustar}.
\end{df}

The reason why we consider \ks is that the stable locus of $\JJ$ is symplectic by a result of Mukai.

\begin{thm}[Theorem 0.1 \cite{Mu}]\label{Muk}
The stable locus of $\JJ$ is smooth and has a symplectic form, given pointwise at a stable sheaf $\FF$ by 
\begin{equation}\label{Mukai}
\xymatrix{
\Ext^1(\FF,\FF) \times \Ext^1(\FF,\FF) \ar[r]^{\qquad \cup} &  \Ext^2(\FF,\FF) \ar[r]^{\quad tr} & H^2(\OO) \ar[r]^{\cdot / \sigma} & \C.}
\end{equation}
\end{thm}

The fact that $S$ is a double cover of $Y$ implies that also the stable locus of $\PP$ is symplectic: 

\begin{thm}\label{Psympl}
\begin{itemize}
\item[$i)$] The stable locus of $\PP$ has a symplectic form, induced by (\ref{Mukai}) by restriction.
\item[$ii)$] $\supp: \PP \to \P^3$ is a Lagrangian fibration and the generic fiber is an abelian 3-fold with polarization type $(1,1,2)$.
\end{itemize}
\end{thm}

\begin{proof}
\begin{itemize}
\item[$i)$] It suffices to prove that $\eta$ preserves the symplectic form $\sigma$, i.e. $\eta^* \sigma =\sigma$. 
We show that $\tau^* \sigma = - \sigma$ and $j \sigma =- \sigma$.  

In the definition of the symplectic structure (\ref{Mukai}), all the identifications are intrinsic except for the last one 
\begin{equation}
\xymatrix{
H^2(\FF) \ar[r]^{\cdot / \sigma} & \C,}
\end{equation}
so the action of $\tau^*$ on $H^{(2,0)}(\JJ^{stable})$ is the same of the action of $\tau$ on $H^{(2,0)}(S)$, and for a \kk double over of a cubic surface $\tau^* \sigma = - \sigma$. 

The action of $j$ on $H^{(2,0)}(\JJ^{stable})$ is described in the proof of Proposition 3.11 of \cite{ASF} (the same argument applies in this case).

\item[$ii)$] The fact that $\supp$ is a \LF is proven in the 2nd case of Section 6 \cite{Mark2}. 
The generic fiber is the Prym variety of a double cover of an elliptic curve by a curve of genus 4, hence the polarization type is $(1,1,2)$ by Section 3 \cite{Mum}. 
\end{itemize}
\end{proof} 

\section{Singularities of $\PP$}

In this section we describe the singular locus and the singularities of $\PP$, and we deduce that it does not admit any symplectic desingularisation.

$\\ $

By Theorem \ref{Muk}, the singular locus of $\JJ$ is contained in the locus of strictly semistable sheaves. 
As a torsion free sheaf on an integral curve is stable with respect to any polarization, a strictly semistable sheaf of $\JJ$ is supported on a non-integral curve of $|\pi^*(-K_{Y})|$. 

Hence the singular locus $Sing(\PP)$ of $\PP$ is contained in the locus of $\eta$-invariant strictly semistable sheaves in $\JJ$, which are supported on non-integral curves of $\pi^*|-K_{Y}|$.

\begin{lm}
The locus of $\eta$-invariant strictly semistable sheaves of $\JJ$ is the union of the 27 singular K3 surfaces given by $\overline{J^{-2}}^{C}$ of the 27 elliptic pencils associated to the 27 lines on $Y$. 
Each \kk has 5 $A_1$-singularities, and each of these singular points is an intersection point of 3 K3 surfaces. $\\ $
A smooth point represents a polystable sheaf of type 
\begin{equation}\label{f1f2}
\FF = \OO_{C_1}(-2) \oplus \FF_2, \quad \FF_2 \in J^{-2}(C_2),
\end{equation}
where $C_1$ is a smooth rational curve, $C_2$ a curve of arithmetic genus $1. \\ $ 
A singular point represents a polystable sheaf of type 
\begin{equation}\label{f1f2f3}
\FF = \OO_{C^1}(-2) \oplus \OO_{C^2}(-2) \oplus \OO_{C^3}(-2), 
\end{equation}
where $C^i$ are smooth rational curves, intersecting transversely in pairs in 2 points. 
\end{lm}

\begin{proof}
Firstly, we determine the strictly semistable sheaves supported on a curve of $\pi^*|-K_{Y}|$. 
A polystable non-stable sheaf $\FF$ of $\JJ$ with $\supp(\FF) \in \pi^*|-K_{Y}|$  has as support a non-integral curve, hence a curve of type $C_1 \cup C_2$ or $C^1 \cup C^2 \cup C^3$ as described in Lemma \ref{blabla}. 

In the case of $C_1 \cup C_2$, $\FF = \FF_1 \oplus \FF_2$, where $\FF_i$ is a pure 1-dimensional sheaf on $C_i$ of degree $d_i$ and $\mu_C({\FF_1})=\mu_C({\FF_2}).$
Since
$$\mu_C({\FF_i})=\frac{\chi (\FF_i)}{C_i \cdot C}=\frac{1-g_i+d_i}{C_i^2 + C_1 \cdot C_2}=\frac{1-g_i+d_i}{2g_i-2 + C_1 \cdot C_2},$$
we get $2d_1+2=d_2.$
Moreover $\chi (\FF) = \chi (\FF_1) + \chi (\FF_2)$ implies $-4=d_1+d_2$.
Hence $d_1=d_2=-2$, and $\FF$ is as in (\ref{f1f2}). 

In the case of $C^1 \cup C^2 \cup C^3$, there are two possibilities:
\begin{itemize}
\item[-] $\FF$ is as in (\ref{f1f2}), where $C_1$ is one of the 3 rational curves and $C_2$ is the union of the remaining 2. 
\item[-] $\FF = \FF_1 \oplus \FF_2 \oplus \FF_3$,  where $\FF_i$ is a pure 1-dimensional sheaf on $C_i$ of degree $d_i$ and $\mu_C({\FF_1})=\mu_C({\FF_2})=\mu_C({\FF_3})$, i.e. $d_1=d_2=d_3$. 
Since $\chi (\FF) = \chi (\FF_1) + \chi (\FF_2) + \chi (\FF_3)$, we have $d_1=d_2=d_3=-2$ and $\FF$ is as in (\ref{f1f2f3}). 
\end{itemize}

Secondly, we show that all the polystable sheaves supported on curves of $\pi^*|-K_Y|$ are $\eta$-invariant. 
Both in (\ref{f1f2}) and in (\ref{f1f2f3}), $\eta$ acts separetely on the addendi of $\FF$, which are points of Prym varieties of double covers of rational curves. 
As the hyperelliptic involution of a hyperelliptic curve induces the involution $-1$ on the Jacobian, the corresponding Prym variety coincides with the Jacobian. 
So $\FF$ is $\eta$-invariant. 

Thirdly, we describe the locus of $\eta$-invariant strictly semistable sheaves using the support map. 
To this aim, we observe that each dual line $L^{\vee}$ gives a pencil of conics on $Y$, and its pullback is a pencil of elliptic curves on $S$. 
In the notation of Lemma \ref{blabla}, the double cover in $S$ of the line $L$ is one of the 27 rational curves $C_1$, and the elliptic pencil is $|C_2|$ (which contains reducible curves of type $C^i \cup C^j$). 
The relative compactified Jacobian $\overline{J^{-2}}^{C}$ of each elliptic pencil gives the sheaves $\FF_2$ in (\ref{f1f2}) (and their degenerations $\OO_{C^i}(-2) \oplus \OO_{C^j}(-2)$ in (\ref{f1f2f3})). 
It is $\JJ_{C_2}^{C,-2}$, hence a K3 surface. 
The elliptic fibration map is the support map restricted to the polystable sheaves corresponding to $C_1$.  
Since $|C_2|$ contains 5 reducible fibers with two simple nodes, and $\chi(S)=24$, $S \to |C_2|$ has 14 irreducible singular members with a simple node. 
Hence the elliptic fibration $\overline{J^{-2}}^{C}$ has exactly 19 irreducible singular fibers and $\chi=19$, so it is singular. 

\end{proof}

To determine the singularities of $\PP$, we describe the tangent cone $C_{[\FF]}(\PP)$ to $\PP$ at the strictly semistable sheaves.  
To this aim, we use the following local analytic model of $\JJ$ (see Section 2.6 and 2.7 \cite{KLS}). 

\begin{thm}\label{Kuran}
There exists a linear map $k: \Ext^1_S(\FF,\FF) \to \Ext^2_S(\FF,\FF)$, called Kuranishi map, satisfying the following properties:
\begin{itemize}
\item[1)] its image is contained in the kernel of the trace map 
$$\tr: \Ext^2_S(\FF,\FF) \to H^2(\OO_S)=\C,$$ 
denoted by  $\Ext^2_S(\FF,\FF)_0$;
\item[2)] $k$ is equivariant with respect to the natural conjugation action of $G:=\PAut(\FF)$ on $\Ext^1_S(\FF,\FF)$ and $\Ext^2_S(\FF,\FF)$;
\item[3)] $k$ has an expansion at $0$ $$k=k_2+k_3+...$$
starting from a quadratic term which is the cup product
$$k_2(\GG):=\frac{1}{2} \GG \cup \GG;$$
\item[4)] $(k^{-1}(0) // G,0)$ is a local analytic model of $(\JJ,\FF)$ and $(k^{-1}_2(0) // G,0)$ is a local analytic model of $C_{[\FF]}(\JJ)$.
\end{itemize}
\end{thm}

\begin{thm}\label{singgen}
At a polystable sheaf $\FF=\FF_1 \oplus \FF_2$ as in (\ref{f1f2}), $C_{[\FF]}(\PP)$ is locally analytically equivalent to $\C^2 \times (\C^4/ \pm 1)$. 
\end{thm}

\begin{proof}
At the infinitesimal level, $\eta$ induces an involution $\eta^*$ on $\Ext^1_S(\FF,\FF)$. 
By Theorem \ref{Kuran}, there is a natural sequence of inclusions 
\begin{equation}\label{inclusions}
C_{[\FF]}(\PP)= C_{[\FF]}(\JJ^{\eta}) \subset C_{[\FF]}(\JJ)^{\eta^*} \subset (\Ext^1_S(\FF,\FF)//G)^{\eta^*}.
\end{equation}
To deduce the thesis, we show that (\ref{inclusions}) is a sequence of identities in the following two steps: 
\begin{itemize}
\item[$i)$] 
\begin{equation}\label{i)}
\Ext^1_S(\FF,\FF) // G = \C^2 \times \widehat{\Imm(\sigma_{2,2})},
\end{equation} 
where $\widehat{\Imm(\sigma_{2,2})}$ is the affine cone over the Segre embedding $\sigma_{2,2}: \P^3 \times \P^3 \to \P^{15}.$
\item[$ii)$] 
\begin{equation}\label{ii)}
(\Ext^1_S(\FF,\FF)//G)^{\eta^*}=\C^2 \times (\C^4/ \pm 1).
\end{equation} 
\end{itemize} 

\begin{itemize}
\item[$i)$]
Set
$$
\begin{aligned} 
\qquad U_1:=\Ext^1_S(\FF_1,\FF_1), \ \ & 
U_2:= \Ext^1_S(\FF_2,\FF_2) , \\ 
\qquad W:= \Ext^1_S(\FF_1,\FF_2), \ \ &
W':= \Ext^1_S(\FF_2,\FF_1).
\end{aligned}
$$
By Serre duality $W'$ is the dual of $W$ with respect to the pairing $\tr( \_ \circ \_)$.
As the supports of $\FF_1$ and $\FF_2$ are transversal, we get $W=\C^4$. 
Since $C_1$ is a rational curve, $U_1=H^0(\NN_{C_1/S}) \oplus \Ext^1_{C_1}(\FF_1,\FF_1)=0$.
Since $C_2$ has genus 1, $\FF_2$ can be deformed either as a sheaf in $J^{-2}(C_2)$ or by varying its support, the two possibilities corresponding to the two summands in $U_2=H^0(\NN_{C_2/S}) \oplus \Ext^1_{C_2}(\FF_2,\FF_2)=\C^2.$ 
So 
\begin{equation}\label{uw}
\Ext^1_S(\FF,\FF)= U_1 \times U_2 \times W \times W^*=\C^{10}.
\end{equation}
Choosing coordinates $x_1,...,x_4$ in $W$ such that $\tau$ exchanges $x_1 \leftrightarrow x_2$ and $x_3 \leftrightarrow x_4$, let $y_1,...,y_4$ be the dual coordinates in $W^*$. 
Let $z_1,z_2$ be coordinates in $U_2$. 

$\Aut(\FF)=\Aut(\FF_1) \times \Aut(\FF_2)={\C^*}^2$ by the stability of $\FF_i$, hence $G=\C^*$. 
Moreover $G$ acts trivially on $U_2$, while on $W \times W^*$  
$$
(\lambda_1,\lambda_2) \cdot (\underline{x},\underline{y})=(\lambda_1 \lambda_2^{-1} \underline{x},\lambda_1^{-1} \lambda_2 \underline{y}) \mbox{ for } (\lambda_1,\lambda_2) \in \Aut(\FF), 
$$
so its action on $\Ext^1_S(\FF,\FF)$ is 
\begin{equation}\label{zxy}
\lambda \cdot (\underline{z},\underline{x},\underline{y})=(\underline{z},\lambda \underline{x},\lambda^{-1}\underline{y}), \mbox{ where } \lambda=\lambda_1/\lambda_2. 
\end{equation}
By (\ref{zxy}) and (\ref{uw})
\begin{equation}\label{//G}
\Ext^1_S(\FF,\FF)//G= U_2 \times (W \times W^*)//G.
\end{equation}
The algebra of invariants of the action of $G$ on $\P(W \times W^*)$ is generated by the quadratic monomials 
\begin{equation}\label{coordu}
u_{ij}=x_i y_j,
\end{equation}
and the generating relations are the quadratic ones 
\begin{equation}\label{quad}
u_{ij}u_{kl}=u_{kj}u_{il}.
\end{equation}
So (\ref{//G}) gives (\ref{i)}).

\item[$ii)$]
Since $C_i$ is $\tau$-invariant for $i=1,2$, we have by (\ref{//G})
\begin{equation} \label{prodotto}
(\Ext^1_S(\FF,\FF)//G)^{\eta^*}=U_1^{\eta^*} \times U_2^{\eta^*} \times ((W \times W^*)//G)^{\eta^*}.
\end{equation}
Considering $\PP_{C_i}$, we get the natural identification 
\begin{equation}\label{tangent}
U_i^{\eta^*}=C_{[\FF_i]}(\PP_{C_i})=T_{[\FF_i]}(\PP_{C_i}),
\end{equation}
because $\FF_i$ is a stable and hence smooth point of $\PP_{C_i}$.
It remains to study the last term of (\ref{prodotto}). 
The involution $j$ exchanges $x_i \leftrightarrow y_i,$ thus the action of $\eta^*$ on $W \times W^*$ is 
$$
\eta^*(\underline{x},\underline{y})=(y_2,y_1,y_4,y_3,x_2,x_1,x_4,x_3).
$$

Its fixed locus is then 
\begin{equation}\label{fixxxx}
x_1=y_2,x_2=y_1,x_3=y_4,x_4=y_3.
\end{equation}

Since 
\begin{equation}\label{etaacts}
\eta^*(\lambda \cdot (\underline{x},\underline{y}))= \frac{1}{\lambda}\eta^*(\underline{x},\underline{y}),
\end{equation}
$\eta^*$ induces a well-defined involution on $(W \times W^*)//G$. 
From (\ref{etaacts}), we see that $\eta^*$ is not $G$-invariant, hence its fixed locus on the quotient cannot be described as the quotient of its fixed locus. 
To characterize it, we observe that in the $G$-invariant coordinates $u_{ij}$ of (\ref{coordu}), (\ref{fixxxx}) becomes 
\begin{equation}\label{these}
u_{11}=u_{22}, u_{33}=u_{44}, u_{13}=u_{42}, u_{14}=u_{32}, u_{23}=u_{41}, u_{24}=u_{31}.
\end{equation}
The quadratic relations (\ref{quad}), combined with (\ref{these}), give the equations of the image of the Veronese embedding $v_2:\P^{3} \to \P^{9}.$ 
Thus 
\begin{equation}\label{luogofisso}
((W \times W^*)//G)^{\eta^*}=\widehat{\Imm(v_2)}
\end{equation}
where $\widehat{\Imm(v_2)}$ is the affine cone over $\Imm(v_2)$. 
Moreover 
\begin{equation}\label{veronese}
\widehat{\Imm(v_2)} = \C^{4} / \pm 1.
\end{equation} 
Indeed, choosing coordinates $w_1,...,w_{4}$ on $\C^{4}$ with the action of $- 1$ given by $w_i \mapsto -w_i$, the algebra of invariant functions has generators 
\begin{equation} \nonumber
v_{ij}:=w_i w_j \mbox{ for } i \leq j,
\end{equation}
and relations 
\begin{equation} \nonumber
v_{ij}v_{kl}=v_{kj}v_{il},
\end{equation} 
which describe exactly $\widehat{\Imm(v_2)}$.

Thus combining (\ref{prodotto}), (\ref{luogofisso}) and (\ref{veronese}), we obtain (\ref{ii)}). 
\end{itemize}
\end{proof}

\begin{cor}\label{nosympl}
$\PP$ does not admit any symplectic desingularization.
\end{cor}

\begin{proof}
By Theorem \ref{singgen}, $\PP$ is locally analytically isomorphic to $\C^2 \times (\C^4/ \pm 1)$ around a polystable sheaf $\FF=\FF_1 \oplus \FF_2$. 
The singularity $\C^4/\pm 1$ is $\Q$-factorial, so it has no small resolutions, and terminal, that is the canonical sheaf of any resolution of singularities contains all the exceptional divisors with strictly positive coefficients.
Thus none of the resolutions has trivial canonical class, and none of them is symplectic.
\end{proof}

\begin{remark}
It is interesting to note that $\JJ$ admits a symplectic resolution while $\PP$ does not.
\end{remark}

\begin{thm}\label{singpar}
At a polystable sheaf $\FF=\OO_{C^1}(-2) \oplus \OO_{C^2}(-2) \oplus \OO_{C^3}(-2)$ as in (\ref{f1f2f3}), $C_{[\FF]}(\PP)$  is locally analytically equivalent to $\C^6 / \Z_2 \times \Z_2$, where the action on $\C^6$ is given by 
\begin{equation}\label{z2z2}
\Z_2 \times \Z_2 = <(1,1,-1,-1,-1,-1),(-1,-1,1,1,-1,-1)>.
\end{equation} 
\end{thm}

\begin{proof}
To deduce the thesis, we show that (\ref{inclusions}) is a sequence of identities in the following two steps 
\begin{itemize}
\item[$i)$] 
\begin{equation}\label{ib)}
\Ext^1_S(\FF,\FF) // G = Z,
\end{equation} 
where $Z$ is the affine cone over the variety described by the equations (\ref{uij01}), (\ref{vklm}), (\ref{wklm}), (\ref{vwu}).

\item[$ii)$] 
\begin{equation}\label{iib)}
(\Ext^1_S(\FF,\FF)//G)^{\eta^*}=\C^6 / \Z_2 \times \Z_2,
\end{equation} 
where the action is as in (\ref{z2z2}).
\end{itemize} 

\begin{itemize}
\item[$i)$] 
We set
$$\begin{aligned} 
W_{12}:=\Ext^1_S(\OO_{C^1}(-2),\OO_{C^2}(-2)), \ \ & 
\quad W_{21}:=\Ext^1_S(\OO_{C^2}(-2),\OO_{C^1}(-2)), \\ 
W_{13}:= \Ext^1_S(\OO_{C^1}(-2),\OO_{C^3}(-2)) , \ \ & 
\quad W_{31}:= \Ext^1_S(\OO_{C^3}(-2),\OO_{C^1}(-2)) , \\
W_{23}:= \Ext^1_S(\OO_{C^2}(-2),\OO_{C^3}(-2)), \ \ &
\quad W_{32}:= \Ext^1_S(\OO_{C^3}(-2),\OO_{C^2}(-2)).
\end{aligned}
$$
By Serre duality, $W_{ij}=W_{ji}^*$. 
Since the supports of $\OO_{C^i}(-2)$ and $\OO_{C^j}(-2)$ are transversal for $i \neq j$, we get $W_{ij}=\C^{C^1 \cdot C^2}=\C^2.$
So
$$\Ext^1_S(\FF,\FF)= W_{12} \times W_{13} \times W_{23} \times W_{12}^* \times W_{13}^* \times W_{23}^* = \C^{12}.$$

Choosing coordinates $x_{ij}^0,x_{ij}^1$ in $W_{ij}$ such that $\tau(x_{ij}^0)=x_{ij}^1$, let $y_{ij}^0,y_{ij}^1$ be the dual ones in $W^*_{ij}$. 

By the stability of $\FF_i$, $\Aut(\FF)={\C^*}^3$, hence $G:=\PAut(\FF)={\C^*}^2$.  Setting $\epsilon_1:=\lambda_1/\lambda_2, \epsilon_2:=\lambda_2/\lambda_3$ for $(\lambda_1,\lambda_2,\lambda_3) \in \Aut(\FF)$, its action on $\Ext^1_S(\FF,\FF)$ is 
$$
(\epsilon_1,\epsilon_2) \cdot (x_{ij}^k,y_{ij}^k)=(\epsilon_1 x_{12}^k, \epsilon_1 \epsilon_2 x_{13}^k,\epsilon_2 x_{23}^k, \epsilon_1^{-1} y_{12}^k, \epsilon_1^{-1} \epsilon_2^{-1} y_{13}^k,\epsilon_2^{-1} y_{23}^k).
$$
The algebra of invariants of the action of $G$ on $\P (\Ext^1_S(\FF,\FF))$ is generated by the 12 quadratic monomials 
\begin{equation}\label{star}
u_{ij}^{kl}:=x_{ij}^k y_{ij}^{l} \qquad i < j,
\end{equation} 
and by the 16 cubic monomials 
\begin{equation}\label{starstar}
v^{klm}:=x_{13}^k y_{12}^l y_{23}^m, \qquad w^{klm}:=y_{13}^k x_{12}^l x_{23}^m.
\end{equation} 

Its generating relations are the 3 equations in $u_{ij}^{kl}$
\begin{equation}\label{uij01} 
u_{ij}^{00} u_{ij}^{11}=u_{ij}^{01} u_{ij}^{10},
\end{equation} 
the 18 equations in $v^{ijk}$, $w^{ijk}$
\begin{align}\label{vklm} 
v^{klm} v^{k'l'm'}=v^{k'lm} v^{kl'm'}=v^{kl'm} v^{k'lm'}=v^{klm'} v^{k'l'm}, \quad \\
\label{wklm}
w^{klm} w^{k'l'm'}=w^{k'lm} w^{kl'm'}=w^{kl'm} w^{k'lm'}=w^{klm'} w^{k'l'm}, 
\end{align} 
and the 64 cubic equations
\begin{align}\label{vwu}
v^{klm} w^{k'l'm'}=u_{13}^{kk'} u_{12}^{ll'} u_{23}^{mm'}.
\end{align}
Hence we deduce (\ref{ib)}).

\item[$ii)$] The action of $j$ is $x_{ij}^k \leftrightarrow y_{ij}^k$, so that
$$
\eta^*(x_{ij}^k,y_{ij}^k)=(y_{12}^1,y_{12}^0,y_{13}^1,y_{13}^0,y_{23}^1,y_{23}^0,x_{12}^1,x_{12}^0,x_{13}^1,x_{13}^0,x_{23}^1,x_{23}^0).
$$
Its fixed locus is then 
\begin{equation}\label{lalala}
y_{12}^1=x_{12}^0,  y_{12}^0=x_{12}^1,  y_{13}^1=x_{13}^0,  y_{13}^0=x_{13}^1,  y_{23}^1=x_{23}^0,  y_{23}^0=x_{23}^1.
\end{equation}

Since
\begin{equation*}
\eta^*((\epsilon_1,\epsilon_2) \cdot (x_{ij}^k,y_{ij}^k))=\left(\frac{1}{\epsilon_1}, \frac{1}{\epsilon_2} \right) \eta^* (x_{ij}^k,y_{ij}^k),\end{equation*}
$\eta^*$ induces a well defined involution on $\Ext^1_S(\FF,\FF)//G$. 
As $\eta^*$ is not $G$-invariant, its fixed locus on the quotient cannot be expressed as quotient of its fixed locus. 
We can describe it using the invariant coordinates $u_{ij}^{kl},v^{klm},w^{klm}$. 
Substituing (\ref{lalala}) in (\ref{star}) and (\ref{starstar}), we see that the fixed locus of $\eta^*$ is 
\begin{align}\label{sisi}
u_{ij}^{00} & = u_{ij}^{11}, \\
\label{sisisi}
v^{klm} & = w^{1-k,1-l,1-m}.
\end{align}

So the function algebra of $(\Ext^1_S(\FF,\FF) // G)^{\eta^*}$ has the 9 coordinate functions $u_{ij}^{00}, u_{ij}^{01}, u_{ij}^{10}$ and the 8 coordinate functions $v^{klm}$ as generators. 
Using (\ref{sisi}) and (\ref{sisisi}), the relations (\ref{uij01}) give the 3 equations
\begin{equation}\label{iaia}
(u_{ij}^{00})^2=u_{ij}^{01} u_{ij}^{10},
\end{equation}
the relations (\ref{vwu}) give the 36 equations
\begin{equation}\label{iaiaia}
v^{klm} v^{k'l'm'} = u_{13}^{k,1-k'} u_{12}^{l,1-l'} u_{23}^{m,1-m'},
\end{equation}
while (\ref{vklm}) and (\ref{wklm}) follow from (\ref{iaia}) and (\ref{iaiaia}).

These equations describe the quotient (\ref{iib)}). 
Indeed choosing coordinates $r^0_1,r^1_1,r^0_2,r^1_2,r^0_3,r^1_3$, in which the action is given by (\ref{z2z2}), the algebra of invariant functions is generated by the 9 quadratic monomials
$$s_i^{jk}:=r^j_i r^k_i$$
and by the 8 cubic monomials 
$$t^{ijk}:=s^i_1 s^j_2 s^k_3,$$
with the 3 quadratic relations
$$s_i^{01}=s_i^{00} s_i^{11},$$
and the 36 cubic ones
$$t^{ijk} t^{i'j'k'} = s_1^{ii'}s_2^{jj'}s_3^{kk'}.$$
\end{itemize} 
\end{proof}

\begin{cor}\label{singloc}
The singular locus of $\PP$ coincides with the locus of $\eta$-invariant strictly semistable sheaves and consists only of quotient singularities.
\end{cor}

\begin{proof}
The identification of the singular locus with the strictly $C$-semistable sheaves and the analysis of the type of singularities follow from Theorem \ref{singgen} and Theorem \ref{singpar}.
\end{proof}

\section{Simple connectedness and irreducibility of $\PP$}
In this section we prove that $\PP$ is simply connected and the $H^{(2,0)}$ of its smooth locus is generated by the symplectic form.
We deduce this describing a birational model of $\PP$ as a quotient of $S^{[3]}$ by an involution.
$ \\ $

\begin{lm}
There exists a rational map
\begin{align} \nonumber
\psi: & S^{[3]} \dashrightarrow \PP \\ \nonumber
& \xi \mapsto (1-\tau)\xi
\end{align}
\end{lm}

\begin{proof}
For a generic $\xi \in S^{[3]}$, $C_{\xi}:=\langle \xi,p_0 \rangle \cap S$ is a generic element of $\pi^*|-K_{Y}|$ and $C'_{\xi}:= \langle \pi (\xi) \rangle \cap Y$ is a generic member of $|-K_{Y}|$. 
Indeed, a $\tau$-invariant hyperplane contains $p_0$, so it is given by $p_0$ and 3 other points.
Hence generically $\xi \in J^3(C_{\xi})$, and $\xi - \tau (\xi) \in P(C_{\xi}, \tau)$.
\end{proof}

\begin{remark}
The indeterminacy locus of $\psi$ is given by $\xi$ such that $\dim \langle \xi \rangle <3$, i.e. $\langle \xi \rangle$ is a line. 
A line meeting $S$ in 3 points, meets also $Z_2$ in 3 points, so it lies on $Z_2$. 
Vice versa, a line on $Z_2$ clearly meets $Z_3$ in 3 points, so it also meets $S$ in these 3 points. 
Hence 
$$\Indet(\psi)=\{ \mbox{lines in } Z_2\}= \P^3.$$
\end{remark}

\begin{lm}[Proposition 4.1 \cite{O4}]
The involution 
\begin{align}
\iota_1: & S^{[3]} \dashrightarrow S^{[3]} \\ \nonumber
& \xi \mapsto ( \langle \xi \rangle \cap S )- \xi.
\end{align} 
is antisymplectic.
\end{lm}

\begin{remark}
$\iota_1$ has the same indeterminacy locus of $\psi$
$$\Indet(\iota_1)=\{ \mbox{lines in } Z_2\}=\P^3.$$
Thus $\iota_1$ induces a regular involution on the blowup $\Bl(S^{[3]})$ of $S^{[3]}$ along the $\P^3$ given by the locus of lines in $Z_2$.

If we consider a generic $\xi \in \Indet(\iota_1)$, then $\langle \xi,p_0 \rangle \cong \P^2$ meets $S$ in 6 points, respectively $\xi$ and $\tau( \xi)$, and $\tau( \xi) \in \Indet(\iota_1)$. 
So also $\tau$ extends to an involution on $\Bl(S^{[3]})$.
\end{remark}

\begin{thm}\label{ghj}
$\psi$ is a rational double cover with involution $\iota_2$, where  $\iota_2$ is the involution induced by $\iota_1 \circ \tau$ on the blowup $\Bl(S^{[3]})$ of $S^{[3]}$ along the locus of lines in $Z_2$. 
Hence $M:=\Bl(S^{[3]})/ \iota_2$ is birational to $\PP$. 
\end{thm}

\begin{proof}
Let $\xi=\{p_1,p_2,p_3\}$ be generic. 
We want to determine all the divisors  $\xi'=\{p'_1,p'_2,p'_3\}$ on $C_{\xi}$ such that $\xi- \tau (\xi) \sim \xi' - \tau(\xi')$. 
Equivalently, setting $\delta:= \xi + \tau(\xi')$ , we want to determine the solutions of $\delta \sim \tau(\delta)$ for $\xi$ generic.

If $\delta= \tau(\delta)$, then $\delta$ is $\tau$-invariant and, modulo the permutations of $\xi$ and of $\xi'$, we have only 3 possibilities:
\begin{itemize}
\item[$a)$] $p_i'= \tau(p_i), \ \ i=1,2,3$, then $2\xi \sim 2 \tau (\xi),$ hence $\xi$ is non-generic.
\item[$b)$] $p_1'=\tau (p_1), \ \ p_2'=\tau (p_2), \ \ p_3'= p_3,$ then $2(p_1+p_2) \sim 2(\tau (p_1)+\tau (p_2)),$ hence $\xi$ is non-generic.
\item[$c)$] $p_1'=\tau (p_1), \ \ p_2'=p_2, \ \ p_3'= p_3,$ then $2p_1 \sim 2\tau (p_1),$ hence $\xi$ is non-generic.
\end{itemize}

If $\delta \neq \tau(\delta),$ then $\dim|\delta|>0$. 
By the Riemann-Roch theorem we have $\dim|\delta|=3+\dim|K_{C_{\xi}}-\delta|$, with $\deg K_{C_{\xi}}=\deg \delta =6.\\ $
There are 3 subcases:
\begin{itemize}
\item[$d)$] $K_{C_{\xi}} \sim \delta$, so $\langle \delta \rangle$ is a plane in $\langle C_{\xi} \rangle \cong \P^3$, and $|\delta|={\P^3}^*$. Then $\tau (\xi')$ is uniquely determined as $\langle \xi \rangle \cap C_{\xi}-\xi$. 
So the unique non-trivial solution is $\iota_2(\xi).$
\item[$e)$] $K_{C_{\xi}} \neq \delta$ and $|\delta|$ is base point free. 
Then none of the possible 5-uples of points of $\delta$ lies on a plane. 
Now $|\OO_{\P^3}(2)|=\P^9$, and $|\OO_{C_{\xi}}(2)|=|2H_{C_{\xi}}| \cong \P^8$, since $C_{\xi}\subset \langle \xi,p_0 \rangle \cap Z_2$. 
So there exist 6 points $\bar{\delta}$ on $C_{\xi}$ such that $|\delta|$ consists of the residual intersection $(q \cap C_{\xi})- \bar{\delta},$ where $q \in |2H_{C_{\xi}}-\bar{\delta}|=\P^2.$ 
Moreover, $\tau$ acts linearly on $\langle \xi \rangle$, so $q \in |2H_{C_{\xi}}-\bar{\delta}|$ if and only if $\tau (q) \in |2H_{C_{\xi}}-\tau(\bar{\delta})|$. 
As $\delta \sim \tau(\delta)$, the two families coincide. 
We deduce that $\bar{\delta}$ is $\tau$-invariant, and hence every quadric in $|2H_{C_{\xi}}-\bar{\delta}|$ is $\tau$-invariant. Thus $\xi=\xi'$, i.e. $\delta=\tau (\delta$), absurd.
\item[$f)$] $K_{C_{\xi}} \neq \delta$, and $|\delta|=\P^2$ has a base point. Then 5 points of $\delta$ span a plane $\Pi$; assume they are $p_1,p_2,p_3,\tau(p_1'),\tau(p_2')$. 
Setting $\bar{p}$ the remaining intersection point of $\Pi$ with $C_{\xi}$, $|\delta|$ is clearly given by $|H_{C_{\xi}}-\bar{p}|=\P^2$. 
As $\delta \sim \tau(\delta)$, $\bar{p}$ is $\tau$-invariant. 
So $\xi$ spans a plane passing through one of the six $\tau$-invariant points of $C_{\xi}$, hence $\xi$ is non-generic.
\end{itemize}

We conclude that the generic fiber of $\psi$ consists of only two points, interchanged by $\iota_2.$
\end{proof}

\begin{cor}\label{h20}
$h^{(2,0)}(\PP)=1$.
\end{cor}

\begin{proof}
As $M$ admits a rational dominant map onto $\PP$ by Lemma \ref{ghj}, we have $h^{(2,0)}(\PP)=h^{(2,0)}(M)=h^{(2,0)}(S^{[3]})=1$.
\end{proof}

\begin{lm}
$\Fix(\iota_2)$ is the disjoint union of two smooth irreducible 4-folds and 120 isolated points.
\end{lm}

\begin{proof}
The fixed locus of a biregular involution on a smooth variety is smooth.

For a generic $\iota_2$-invariant $\xi$, the plane $\langle \xi \rangle$ is $\tau$-invariant, because the planes are $\iota_1$-invariant.
Since Fix($\tau$)$=H_4 \cup p_0$, either $\langle \xi \rangle \subset H_4$ or $p_0 \in \langle \xi \rangle$ (indeed if there exists $p \in \langle \xi \rangle-H_4$, then $p_0 \in \langle \xi \rangle$).

In the first case $\xi \subset S \cap H_4$ is $\tau$-invariant, so $\iota_1(\xi)=\xi$. 
Then $\langle \xi \rangle$ is totally tangent to the curve $S \cap H_4$. 
This imposes 3 conditions in $\P^3$, hence we expect a finite number of such $\xi'$s. 
These correspond to the odd theta characteristics of the curve, which are exactly $2^3(2^4-1)=120$.

In the second case, we obtain the remaining part of Fix($\iota_2$) as 
$$\Sigma:=\{ \xi \in S^{[3]}: p_0 \in \langle \xi \rangle \}.$$
To describe it, we consider the natural map $\Bl(S^{[3]}) \to \G(2,4)$ extending the rational map $\xi \mapsto \langle \xi \rangle$.
It is a $\binom{6}{3}=20$-to-1 covering.
The image of $\Sigma$ is clearly 
$$\{\Pi \mbox{ plane } \subset \P^4 : p_0 \in \Pi \} = \sigma_{1,1} \cong \G(1,3),$$ 
so $\Sigma$ is a 20:1 covering of a smooth quadric of $\P^5$. 
Since $\langle \xi \rangle$ is $\tau$-invariant,  $\langle \xi \rangle \cap S= \{\xi, \tau(\xi)\}$. 
The 12 triples $\{p_i,\tau({p_i}) ,p_j\}$ and $\{p_i,\tau({p_i}),\tau(p_j)\}$ sweep a 4-fold $\Sigma_1 \subset S^{[3]}$ which is a double covering of $Y^{[2]}$.
The other 8 triples sweep $\Sigma_2$, an 8-sheeted covering of $\sigma_{1,1}$.
So $\Sigma$ has 2 disjoint irreducible components, $\Sigma_1$ and $\Sigma_2$. 
\end{proof}

\begin{remark}
Considering $\iota_2$ as a rational involution on $S^{[3]}$, it has the same fixed locus, because a line on $Z_2$ does not lie in $H_4$ and does not pass through $p_0$. 
\end{remark}

\begin{thm}\label{pi1}
$\PP$ is simply connected.
\end{thm}

\begin{proof}
As $M$ is a rational double cover of $\PP$ by Lemma \ref{ghj}, $M$ has the same fundamental group as $\PP$. 
Since there are fixed points, $M$ has the same fundamental group as $\Bl(S^{[3]})$, which has the same fundamental group of $S^{[3]}$ because it is a blowup along a smooth locus.
\end{proof}

\begin{remark}
Considering the invariant part of the action of $\eta_*$ on $T_{P}(S^{[3]})$, it is easy to see that $M$ has singularities of type $\C^4 \times (\C^2/ \pm 1)$ on $\Sigma_i$ and of type $\C^6/ \pm 1$ at the isolated points. 
Since the singularities of $M$ and of $\PP$ are different, they are only birational. 
It is interesting to determine explicitly a birational transformation between them.
\end{remark}

\begin{remark}
It remains an open question if $\PP$ can be expressed as a quotient of another manifold by a regular finite group action. 
\end{remark}

\section{Euler characteristic of $\PP$}
In this section we calculate the Euler characteristic of $\PP$, using the fibration structure.  
$\\ $ 

Since the Euler characteristic is additive, i.e. $\chi(X)= \chi(U)+ \chi (X-U)$ for any open $U \subset X$, and multiplicative for topologically trivial fibrations, i.e. $\chi(X \times Y) = \chi(X) \cdot \chi(Y)$, we can stratify $\pi^*|-K_{Y}|$ depending on the fibers. 
Since $\chi=0$ for any smooth abelian variety, it is enough to study the locus of singular fibers, which we call the discriminant of the fibration and we denote by $\Delta$. 

\begin{lm}
The discriminant $\Delta \subset \pi^*|-K_{Y}|$ consists of two irreducible components: the dual $Y^*$ of the cubic surface, of degree 12, and the dual $B^*$ of the branch locus of $\phi$, of degree 18.
\end{lm}

\begin{proof}
$C$ is singular if and only if $C'$ is singular or $C'$ is tangent to $B$. 

The degrees of $Y^*$ and $B^*$ can be easily determined using Schubert calculus in ${\P^3}^*$:
$$\deg Y^* = Y^* \cap \sigma_{1,0,0}^2 = Y^* \cap \sigma_{1,1,0} = \{ H: l \subset H, H \mbox{ tangent to }  Y \},$$
$$\deg B^*= \{ H: l \subset H, H \mbox{ tangent to }  B \},$$
with $l$ a generic line in $\P^3$.
Denoting $P,Q,R$ the intersection points of $Y$ and $l$, we have a natural map $f: \Bl_{P,Q,R}(Y) \to \sigma_{1,1,0} \cong \P^1$ such that $f(p)=\langle p,l \rangle$.
The degree of $Y^*$, which is the number of planes tangent to $Y$ and passing through $l$, corresponds to the number $N$ of singular fibers of $f$.
Using the good properties of the Euler characteristic, we get
$$\chi(\Bl_{P,Q,R}(Y))= \chi(\P^1-N \mbox{ pts}) \chi(\mbox{smooth fiber}) + \chi(N \mbox{ pts}) \chi(\mbox{singular fiber}),$$
i.e. $N=12$, because $\chi(\Bl_{P,Q,R}(Y)) = \chi(\Bl_{9 pts}(\P^2))=\chi(\P^2)+9=12$ and a smooth fiber has $\chi$ equal to zero (it is an elliptic curve), while a singular fiber has $\chi=1$ (it is a nodal plane cubic).
$\\ $To determine $\deg B^*$, we can consider the 6:1 cover $g: B \to \sigma_{1,1,0} \cong \P^1$ such that $g(p)=\langle p,l \rangle$. 
The degree corresponds to the degree of the branch locus, which is 18 by the Riemann-Hurwitz theorem.
\end{proof}

\begin{remark}
The degree of the discriminant locus of $\PP$ is 30. 
For irreducible symplectic 6-folds obtained as Beauville-Mukai integrable systems, the degree is 36. 
General results on the degree of a Lagrangian fibration with Jacobians of integral curves as fibers have been obtained by Sawon in \cite{S1}. 
\end{remark}

\begin{thm}\label{singcurves}
$\Delta$ admits a natural stratification in singular loci (with several irreducible components), corresponding to all the possible singular members of $\pi^*|-K_{Y}|$, as described in the following:
$\\ $

\underline{Dimension 2}

\begin{itemize}

\item[a)] C has a simple $\tau$-invariant node, if $\langle C'\rangle \in B^*-Sing(B^* \cup Y^*)$ (i.e. $C'$ is tangent to B); 
\item[b)] C has two simple nodes, interchanged by $\tau$, if $\langle C'\rangle \in Y^*-Sing(B^* \cup Y^*)$ (i.e. $C'$ has a simple node);

\underline{Dimension 1}

\item[c)] C has two simple $\tau$-invariant nodes, if $\langle C'\rangle$ lies in the complement of $Sing(Sing(B^* \cup Y^*))$ inside the irreducible component of $Sing(B^*)$ corresponding to $C'$ bitangent to B;
\item[d)] C has one cusp, if $\langle C' \rangle$ lies in the complement of $Sing(Sing(B^* \cup Y^*))$ inside the irreducible component of $Sing(B^*)$, which corresponds to $C'$ having a point of triple contact with B;
\item[e)] C has three simple nodes, one fixed by $\tau$ and the others interchanged by $\tau$, if $\langle C' \rangle$ lies in the complement of $Sing(Sing(B^* \cup Y^*))$ inside the irreducible component of $B^* \cap Y^*$ corresponding to $C'$ tangent to $B$ and having a simple node;
\item[f)] C has a tacnode, if $\langle C'\rangle$ lies in the complement of $Sing(Sing(B^* \cup Y^*))$ inside the irreducible component of $B^* \cap Y^*$ corresponding to $C'$ with a simple node on B, in other words $C'$ is cut out by a plane tangent to $Y$ at a point of B;
\item[g)] C has two simple cusps, interchanged by $\tau$, if $\langle C' \rangle$ lies in the complement of $Sing(Sing(B^* \cup Y^*))$ inside the irreducible component of $Sing(Y^*)$ corresponding to $C'$ with a cusp;
\item[h)] C has two irreducible components meeting in four points, interchanged in pairs by $\tau$, if $\langle C'\rangle$ lies in the complement of $Sing(Sing(B^* \cup Y^*))$ inside the irreducible component of $Sing(Y^*)$, which corresponds to $C'$ being a reducible plane cubic that decomposes into a conic and a line;

\underline{Dimension 0}

\item[i)] C has a cusp and a simple node, if $\langle C' \rangle$ is one of the points of $Sing(Sing(B^*))$ corresponding to $C'$ with a triple contact point with B and another simple tangency point; 
\item[j)] C has a tacnode, if $\langle C' \rangle$ is one of the points of $Sing(Sing(B^*))$ corresponding to $C'$ with a quadruple contact point with B; 
\item[k)] C has three simple $\tau$-invariant nodes, if $\langle C'\rangle$ is one of the points of $Sing(Sing(B^*))$ corresponding to $C'$ with a tritangent to B;
\item[l)] C has two simple cusps interchanged by $\tau$ and a simple $\tau$-invariant node, if $C'$ is one of the points of $B^* \cap Sing(Y^*)$, which corresponds to $C'$ being tangent to B and having a simple cusp outside B; 
\item[m)] C has an $A_5$-singularity, if $\langle C'\rangle$ is one of the points of $B^* \cap Sing(Y^*)$ corresponding to $C'$ with a simple cusp on B; 
\item[n)] C has two irreducible components meeting in pairs in two points interchanged by $\tau$ and a simple node only on one of the two components, if $\langle C'\rangle$ is one of the points of $B^* \cap Sing(Y^*)$, which corresponds to $C'$ being a reducible plane cubic that is the union of a line and a conic tangent to B;
\item[o)] C has two $\tau$-invariant simple nodes and two simple nodes interchanged by $\tau$, if $\langle C' \rangle$ is one of the points of $Sing(B^*) \cap Y^*$ corresponding to $C'$ bitangent to B and having a simple node outside B; 
\item[p)] C has a tacnode and a simple node, if $\langle C'\rangle$ is one of the points of $Sing(B^*) \cap Y^*$ corresponding to $C'$ tangent to B and having a singular point on B;
\item[q)] C has a $D_4$-singularity, if $\langle C'\rangle$ is one of the points of $Sing(B^*) \cap Y^*$ corresponding to $C'$ tangent to B in a singular point; 
\item[r)] C has two simple nodes interchanged by $\tau$ and a simple $\tau$-invariant  cusp, if $\langle C'\rangle$ is one of the points of $Sing(B^*) \cap Y^*$ corresponding to $C'$ with a triple tangency point on B and a singular point outside B;
\item[s)] C has three irreducible components meeting in pairs in two points interchanged by $\tau$, if $\langle C'\rangle$ is one of the points of $Sing(Sing(Y^*))$ (i.e. a reducible plane cubic given by three lines).
\end{itemize}
\end{thm}

\begin{proof}
To describe the singularities of $C$ it is enough to look at $C'$. 
If $C'$ has a simple node/cusp outside $B$, then $C$ inherits two nodes/cusps interchanged by $\tau$.
If $C'$ has a double/triple/quadruple tangency point with $B$, then $C$ inherits a $\tau$-invariant simple node/cusp/tacnode.
If $C'$ has a simple node on $B$, then $C$ has a tacnode, because locally $C'$ has equation $u^2+v^2=0$, hence $C$ is given by $t^2=u, u^2+v^2=0$, or $t^4+v^2=0$.
If $C'$ has a simple node on $B$ and $B$ is tangent to one of the two branches of the curve through it, then $C$ has a $D_4$-singularity, because locally $C'$ has equation $uv+v^3=0$, hence $C$ is given by $t^2=u, uv+v^3=0$, or $(t^2+v^2)v=0$.
If $C'$ has a simple cusp on $B$, then $C$ has an $A_5$-singularity, because locally $C'$ has equation $u^3+v^2=0$, hence $C$ is given by $t^2=u, u^3+v^2=0$, or $t^6+v^2=0$.

\end{proof} 

We denote by $\Pi_{\bullet}$ the locus of points such that the condition $\bullet)$ of Theorem \ref{singcurves} holds, and by $\bar{P}_{\bullet}$ the fiber over a point of $\Pi_{\bullet}$ (i.e. the compactified Prym variety of a curve from $\Pi_{\bullet}$). 
Then
\begin{equation}\label{euler}
\chi(\PP)=\chi(\Pi_a) \chi(\bar{P}_a)+...+\chi(\Pi_s) \chi(\bar{P}_s).
\end{equation}

To calculate $\chi(\bar{P}_{\bullet})$, we follow the stratification of $\bar{P}_{\bullet}$ used in \cite{MT}, based on a description of $\bar{J}(C)$ in Chapter 5 \cite{C}. 
$\bar{J}(C)$ admits a stratification in smooth strata whose codimension is equal to the index $i(\FF)$ of the sheaves $\FF$ represented by points of these strata. 
The normalization map $\nu: \tilde{C} \to C$ factorizes through a partial normalization $\bar{\nu}:\bar{C} \to C$ such that $\bar{\nu}^*(\FF)/(tors)$ is invertible, and $i(\FF)$ is the minimum of $length(\bar{\nu}_*(\OO_{\bar{C}})/\OO_{C})$. 
When $C$ is integral, the index takes values between 0 and $\delta(C)=length(\bar{\nu}_*(\OO_{\tilde{C}})/\OO_{C})=p_a(C)-g(C)$.
Each stratum is an extension of $J(\bar{C})$ by an algebraic group.
Let $J_i(C)$ be the stratum of codimension $i$.
So $J_0(C)=J(C)$. 
The map $\FF \mapsto \nu^*(\FF)/(tors)$, restricted to $J_i(C)$, gives a morphism $v_i:J_i(C) \to \Pic^{-i}(\tilde{C})$. 

We denote by $P_i$ the stratum $J_i(C) \cap \bar{P}$ induced on $\bar{P}$. 
So $P_0=P(C,\tau)$, an algebraic group of dimension 4. 
Moreover, $\tau$ extends to an involution on $\tilde{C}$ corresponding to the double cover $\tilde{C} \to \tilde{C}'$, where $\tilde{C}'$ is the normalization of $C'$. 
Each stratum is an extension of $P(\tilde{C},\tau)$ by an algebraic group.

\begin{thm}\label{ecchila}
$\chi(\PP)=2283.$
\end{thm}

\begin{proof}
Similarly to Proposition 2.2 \cite{B2}, $\chi(P_{\bullet})$ corresponds to the number of 0-dimensional strata of $P_{\bullet}$. 
Hence it is non-zero only in the cases $k),n),o),s)$, and it suffices to determine the cardinality of $\Pi_{\bullet}$ and the 0-dimensional strata of $P_{\bullet}$ in these cases.

\begin{itemize}

\item[$k)$] The number of tritangents to $B$ corresponds to the number of odd theta characteristics. 
Hence we get $2^3(2^4-1)=120$ points.

We need to determine the zero-dimensional strata of $\bar{P}$, which is $P_3$, because $C$ is irreducible and $\delta(C)=3$. 
First, we observe that $P(\tilde{C},\tau)$ is given by two points.
Indeed $C$ has three $\tau$-invariant nodes, $p_1,p_2,p_3$. 
By the Riemann-Hurwitz theorem, the induced $\tau$ on $\tilde{C}$ is base-point-free, so $\tau(p_i')=p_i''$, with $\nu^{-1}(p_i)=\{p_i',p_i''\}$, and $\tau$ is a translation by a 2-torsion point $q=[p_1'-p_1'']=[p_1'-p_2'']=[p_3'-p_3''] \in J(\tilde{C})=\tilde{C}$. 
So $\eta$ has 4 fixed points on $\tilde{C}$ (the four solutions of $2p=q$) and $P(\tilde{C},\tau)$ consists of two points.

Following Cook \cite{C}, the elements of $J_3(C)$ are of the form $\nu_*(\LL)$, with $\LL \in \Pic^{-3}(\tilde{C})$. 
To determine $P_3$, we need to describe the action of $\eta$ on $J_3(C)$: $$j(\nu_*(\LL))=\nu_*((\LL^{-1})(-p_1'-p_1''-p_2'-p_2''-p_3'-p_3'')),$$
$$\tau(\nu_*(\LL))=\nu_*(\tau(\LL)).$$
Hence $\nu_*(\LL) \in P_3$ if and only if $\LL \in P(\tilde{C},\tau)$, and $P_3$ consists of two points.

\item[$n)$] The number of reducible curves given by a conic tangent to $B$ and a line on $Y$ does not correspond to the intersection number of $h)$ and $a)$, because the line is not generic. 
To calculate it, we consider the pencil of planes of $\P^3$ containing a fixed line on $Y$. 
Then $B$ intersects the line in 2 points, and a plane of the pencil in the same 2 points plus other 4 points. 
Thus we get a 4:1 cover $B \to \P^1$, and the degree of the branch locus is 14 by the Riemann-Hurwitz theorem. 
As there are 27 lines on a cubic surface, the number of points of $n)$ is $14 \cdot 27= 378$.

The zero-dimensional stratum of $\bar{P}$ is $P_5$, because $C$ has four simple nodes. 
$P(\tilde{C},\tau)$ is a point, because $\tilde{C}$ and $\tilde{C}'$ are rational curves.
The elements of $J_5(C)$ are of the form $\nu_*(\mathcal{O_{\tilde{C}}}(d-1) \oplus \OO_{\tilde{C}}(-d))$, for $d$ satisfying semistability conditions, i.e. $d=0,\pm 1,\pm 2$ ($\pm 1$ represent the same $\SS$-equivalence class, and also $\pm2$). $P_5$ thus consists of three points.

\item[$o)$] To determine the number of nodal curves bitangent to $B$, we calculate it indirectly, determining the degree of the curve (case $c)$ ) of bitangents to $B$ and the number (case $p)$ ) of curves tangent to $B$ and having a singular point on $B$. 
Indeed, $ b) \cap c) = 2 p) + o)$, since $b)$ and $c)$ meet transversely and $p)$ has intersection multiplicity 2 because the bitangents of $b)$ can acquire a node in one of the two tangency points.

The degree of $c)$ can be obtained considering a projection of $B$ onto a plane from a generic fixed point: the number of bitangents of $B$ corresponds to the number of bitangents of the image $B'$, which is a plane curve with the same geometric genus and degree, i.e. $g=4$ and $d=6$. 
Since the arithmetic genus of a plane sextic is 10, $B'$ has 6 simple nodes. By Pl\"ucker formulas, we have  
$$g=(d^*-1)(d^*-2)/2-b-f, \qquad d=d^*(d^*-1)-2b-3f,$$
where $d^*$ is the degree of the dual curve of $B'$, $b$ is the number of bitangents, $f$ the number of flexes.
So we need to determine $d^*$. Again by Pl\"ucker formulas
$$d^*=d(d-1)-2 \delta -3 \kappa,$$
where $\delta$ is the number of simple nodes and $\kappa$ the number of cusps.
Hence $d^*=18$ and $b=90$, so $c)$ has degree 90.

The degree of $p)$ can be obtained considering the curve of the case $f)$. From it we can define 
$$D:=\{ \Pi \cap B - \{p\}: \Pi \mbox{ tangent to } Y \mbox{ at } p \}_{p \in B}.$$
It is a 4:1 cover of $B$ with branching $p)$. 
By the Riemann-Hurwitz theorem, to calculate $p)$, it is enough to determine the genus of $D$. $D$ can be seen as a subvariety of $B \times B \subset \P^3 \times \P^3$: the equation $\sum x_i \partial_i F(\underline{p})=0$, with $((x_i),(\underline{p})) \in B \times B$, gives $D+2\Delta_B$. 
Setting $f_1:=B \times pt$, $f_2:= pt \times B$, we get that $D \sim 12 f_1 + 6 f_2 - 2 \Delta_B$ numerically, hence
$K_D \sim (K_{B \times B}+D)D \sim (18 f_1 +12 f_2 - 2 \Delta_B)(12 f_1+6f_2-2 \Delta_B)$
and using the intersection relations $\Delta_B \cdot f_1=\Delta_B \cdot f_2=0, f_1^2=f_2^2=0, \Delta_B^2=\deg \mathcal{N}_B= \deg \mathcal{T}_B=2-2g_B=-6$, we have $K_D \sim 132.$
So $g(D)=67$, and by the Riemann-Hurwitz theorem the branch locus consists of 108 points.

In conclusion, $o)$ consists of $90 \cdot 12 - 2 \cdot 108=864$ points.
The zero-dimensional stratum of $\bar{P}$ is $P_4$, because $C$ is irreducible and $\delta(C)=4$. 
$P(\tilde{C},\tau)$ is a point, because $\tilde{C}$ and $\tilde{C}'$ are rational curves.
Similarly to $k)$, the elements of $J_4(C)$ are of the form $\nu_*(\mathcal{O_{\tilde{C}}}(-4))$, and $P_4=P(\tilde{C},\tau)$ is a point.

\item[$s)$] We have 45 points, which are the intersection points of the orthogonal lines to the 27 lines on $Y$.

\end{itemize}

Collecting the previous calculation, we obtain by (\ref{euler})

$$\chi(\PP)=120 \cdot 2 + 378 \cdot 3+ 864 \cdot 1 + 45 \cdot 1= 2283.$$
\end{proof}

\begin{remark}\label{error}
Many computations of this proof are similar to those of Proposition 4.3 \cite{MT}. 
In particular, we remark that at point $iv)$ there is a mistake: indeed $P_2$ consists of 2 points, not 4 (and analogously in $P_1$ and $P_0$ there are half of the copies of $\C^*$ and of $\C^* \times \C^*$). 
This follows from the same considerations as in the previous proof for the item $k)$.
For this reason, the computation of the Euler characteristic of the 4-fold from the Del Pezzo of degree $2$ (Proposition 5.1 \cite{MT}) should be corrected as follows: 
$$\chi(\PP)=28 \cdot 2 + 128 \cdot 1+ 28 \cdot 1=212.$$
This computation agrees with the one done by Menet in Proposition 2.40 \cite{Me}, where he determines the Euler characteristic of the 4-fold relating it to the quotient of a K3 surface by an involution.
\end{remark}

\newpage

\bibliographystyle{amssort}

\end{document}